\documentclass[11pt]{amsart}

\usepackage[utf8]{inputenc}

\usepackage{enumerate}

\usepackage{amsmath,amscd,amssymb, amsthm}
\usepackage{enumerate}

\newcommand{\ZZ}{\ensuremath{\mathbb{Z}}}

\newcommand{\RR}{\ensuremath{\mathbb{R}}}
\newcommand{\CC}{\ensuremath{\mathbb{C}}}

\newcommand{\sgn}{\ensuremath{\mathrm{sgn}}}

\newcommand{\Mat}[2]{\ensuremath{\mathrm{M}_{#1}(#2)}}
\newcommand{\MatT}[2]{\ensuremath{\mathrm{M}^\mathrm{T}_{#1}(#2)}}

\newcommand{\GL}[1]{\ensuremath{\mathrm{GL}_{#1}}}
\newcommand{\SL}[1]{\ensuremath{\mathrm{SL}_{#1}}}
\newcommand{\Sp}[1]{\ensuremath{\mathrm{Sp}_{#1}}}

\newcommand{\U}[1]{\ensuremath{\mathrm{U}_{#1}}}
\newcommand{\SU}[1]{\ensuremath{\mathrm{SU}_{#1}}}

\newcommand{\tr}{\ensuremath{\mathrm{tr}}}

\newcommand{\td}{\tilde}

\newcommand{\thup}{\ensuremath{\textrm{th}}}

\newcommand{\ttu}{\ensuremath{\text{\tt u}}}
\newcommand{\ttv}{\ensuremath{\text{\tt v}}}

\newcommand{\sk}{\ensuremath{\mathrm{sk}}}

\newcommand{\hypergeometric}[2]{\ensuremath{{}_{#1} {\rm F}_{#2}}}

\newcommand{\HH}{{\mathbb H}}
\newcommand{\N}{{\mathbb N}}
\newcommand{\Z}{{\mathbb Z}}

\newcommand{\R}{{\mathbb R}}
\newcommand{\C}{{\mathbb C}}
\newcommand{\PP}{{\mathbb P}}

\newcommand{\ra}{\rightarrow}

\newcommand{\trans}[1]{\,^t\hspace{-0.1em}{#1}}
\newcommand{\rank}{\ensuremath{\mathop{\rm rank}\,}}

\newcommand{\rmJ}{\ensuremath{{\rm J}}}

\newtheorem{thm}{\indent \bf Theorem}
\newtheorem{prop}{\indent \bf Proposition}
\newtheorem{cor}{\indent \bf Corollary}
\newtheorem{lemma}{\indent \bf Lemma}

\newtheorem{defn}{\indent \bf Definition}
\newtheorem*{remark}{\indent \bf Remark}
\newtheorem*{remarks}{\indent \bf Remarks}
\newcommand{\ignore}[1]{\relax}

\title[Kohnen's limit process]{Kohnen's limit process for \\ real-analytic Siegel modular forms}

\author{Kathrin Bringmann}
\address{Mathematisches Institut, Universit\"at K\"oln\\ Weyertal 86-90, D-50931 K\"oln\\Germany}
\email{kbringma@math.uni-koeln.de} 

\author{Martin Raum}
\address{Max Planck Institut f\"ur Mathematik\\Vivatsgasse 7\\53111 Bonn, Germany}
\email{MRaum@mpim-bonn.mpg.de} 

\author{Olav K. Richter}
\address{Department of Mathematics\\University of North Texas\\ Denton, TX 76203\\USA}
\email{richter@unt.edu}

\thanks{The first author was partially supported by the Alfried Krupp Prize for Young University Teachers of the Krupp Foundation and by NSF grant DMS-$0757907$. The second author holds a scholarship from the Max Planck society. The third author was partially supported by Simons Foundation Grant $\#200765$}


\begin{document}
\begin{abstract}
Kohnen introduced a limit process for Siegel modular forms that produces Jacobi forms.  He asked if there is a space of real-analytic Siegel modular forms such that skew-holomorphic Jacobi forms arise via this limit process.  In this paper, we initiate the study of harmonic skew-Maass-Jacobi forms and harmonic Siegel-Maass forms.  We improve a result of Maass on the Fourier coefficients of harmonic Siegel-Maass forms, which allows us to establish a connection to harmonic skew-Maass-Jacobi forms.  In particular, we answer Kohnen's question in the affirmative.
\end{abstract}

\maketitle


\section{Introduction}

Jacobi forms occur in the Fourier expansion of Siegel modular forms of degree $2$, 
a fact that played an important part in the proof of the Saito-Kurokawa conjecture 
(see Maass \cite{Maass-Spezial-I, Maass-Spezial-II, Maass-Spezial-III}, 
Andrianov \cite{A-Saito-Kurokawa}, Zagier \cite{Z-SK}, and Eichler and Zagier \cite{EZ}).  
The theory of Jacobi forms has grown enormously since then leading to beautiful applications 
in many areas of mathematics and physics.  Several of these applications rely on real-analytic 
Jacobi forms, and it has been necessary to investigate such forms in detail (see Skoruppa \cite{Sko-mpim, Sko-Invent90}, Berndt and Schmidt \cite{BS}, 
Pitale \cite{Ameya}, the first and third author \cite{B-R-Maass-Jacobi}, and more recently \cite{BRR-H-Harmonic}).  For instance, the real-analytic Jacobi forms in Zwegers \cite{Zwe-thesis} are examples of harmonic Maass-Jacobi forms \cite{BRR-H-Harmonic} that are absolutely vital to the theory of mock theta functions.  
These real-analytic Jacobi forms also impact the theory of Donaldson invariants of $\C\PP^2$ that are related to gauge theory 
(see for example G{\"o}ttsche and Zagier \cite{GZ-Selecta98},  G{\"o}ttsche, Nakajima, and Yoshioka \cite{GNY-DiffGeom08}, 
and Malmendier and Ono \cite{Ono-Mal}), and they emerge in recent work on the Mathieu moonshine (see for example Eguchi and Ooguri and Tachikawa \cite{Egu-Oog-Tac}).

The interplay of holomorphic Jacobi forms and holomorphic Siegel modular forms is well understood, but the analogous situation for real-analytic forms is still mysterious and only partial progress has been made.  For example, current work of Dabholkar, Murthy, and Zagier \cite{DMZ} on quantum black holes and mock modular forms features mock Jacobi forms (which can also be viewed as holomorphic parts of harmonic Maass-Jacobi forms \cite{BRR-H-Harmonic,B-R-Maass-Jacobi}) that occur as Fourier coefficients of meromorphic Siegel modular forms.  Kohnen \cite{Ko-problems, Ko-poincare} suggests a completely different approach to connect real-analytic Jacobi forms and Siegel modular forms.  We use Kohnen's approach to shed more light on 
the relation of Jacobi forms and Siegel modular forms in the real-analytic world.  Let $F$ be a real-analytic Siegel modular form of degree $2$ 
with Fourier-Jacobi expansion 
\begin{gather}
\label{eq:fourierjacobiexpansion}
F(Z)=\sum_{m\in\Z} \phi_m(\tau,z,y')\, e^{2 \pi i\, mx'}
\text{,}
\end{gather}
where throughout the paper, $Z=\left(\begin{smallmatrix} \tau & z \\ z & \tau'\end{smallmatrix}\right)\in\HH_2$ (the Siegel upper half space of degree $2$) with $\tau=x+iy$, $z=u+iv$, and $\tau'=x'+iy'$.  In general, $\phi_m$ is not a Jacobi form due to the dependence on $y'$.  However, in the special case that $F$ in (\ref{eq:fourierjacobiexpansion}) is Maass' \cite{Maass-diff} nonholomorphic Siegel Eisenstein series of degree $2$ and of type $(\frac12, k-\frac12)$, Kohnen \cite{Ko-problems, Ko-poincare} employs the limit
\begin{equation}
\label{Kohnen's limit}
\mathcal{L}(\phi_m):=
\underset{\delta\ra\infty}{\lim}\,e^{\frac{\delta}{2}\,}e^{2\pi m\frac{v^2}{y}}\phi_m\left(\tau,z,\frac{\delta}{4\pi m}+\frac{v^2}{y}\right)\hspace{5ex} (m>0)
\end{equation}
to produce skew-holomorphic Jacobi forms of weight $k$ and index $m$.  Naturally, he asks if there is a space of real-analytic Siegel modular forms such that the limit (\ref{Kohnen's limit}) always yields skew-holomorphic Jacobi forms.  Note also that if $F$ is a holomorphic Siegel modular form of weight $k$, then (\ref{Kohnen's limit}) 
gives precisely the $m$-th Fourier-Jacobi coefficient of $F$, i.e., a holomorphic Jacobi form of weight $k$ and index $m$.

In this paper, we consider the space $\widehat{\mathbb{M}}_k$ of harmonic Siegel-Maass forms of weight $k$ (see Definition \ref{HSM}), which are real-analytic Siegel modular forms of degree $2$ and of type $(\frac12, k-\frac12)$ that are annihilated by the matrix-valued Laplace operator $\Omega_{\frac{1}{2}, k-\frac{1}{2}}$ (defined in (\ref{Laplace})).  
Recall that in the degree one case, Bruinier's and Funke's \cite{B-F-Duke04} operator $\xi_k$ maps harmonic weak Maass forms of weight $k$ to weakly-holomorphic modular forms of weight $2-k$, and the kernel of the map $\xi_k$ consists of weakly-holomorphic modular forms of weight $k$.  In (\ref{Siegel-xi}) we define the corresponding operator $\xi_{\frac{1}{2}, k-\frac{1}{2}}^{(2)}$ for Siegel-Maass forms, which provides a duality between the weights $k$ and $3-k$ (analogous to the situation of the Jacobi forms in Section \ref{sec:skew-Jacobi} and in \cite{B-R-Maass-Jacobi}), and forms in the kernel are analogs of ``holomorphic'' Siegel-Maass forms.  In Section \ref{sec:skew-Jacobi}, we introduce the space $\widehat{\mathbb{J}}_{k,m}^{\sk}$ of harmonic skew-Maass-Jacobi forms of weight $k$ and index $m$ (see Definition \ref{sMJ}), which contains the space $\rmJ_{k,m}^{\sk}$ of skew-holomorphic Jacobi forms of weight $k$ and index $m$.  We use Kohnen's limit process to prove the following theorem, which connects $\widehat{\mathbb{J}}_{k,m}^{\sk}$ and $\widehat{\mathbb{M}}_k$, and in particular, answers Kohnen's question.  Throughout this paper we assume that $k$ is an odd integer such that $k\ne 1,3$.

\vspace{1ex}

\begin{thm}
\label{main}
Let $F \in \widehat{\mathbb{M}}_k$ with Fourier-Jacobi expansion as in (\ref{eq:fourierjacobiexpansion}), and if $k>3$ assume that $\xi_{\frac{1}{2}, k-\frac{1}{2}}^{(2)}(F)=0$.  Let $m > 0$.  If $k>3$, then $y^{\frac{1}{2} - k\,} \mathcal{L}\bigl(\det Y^{k - \frac{1}{2}\,} \phi_m\bigr)\in \rmJ_{k,m}^{\sk}$, and if $k<0$, then $y^{\frac{1}{2} - k\,} \mathcal{L}\bigl(\det Y^{k - \frac{1}{2}\,} \phi_m\bigr)\in\widehat{\mathbb{J}}_{k,m}^{\sk}$.
\end{thm}

\vspace{1ex}


\vspace{1ex}

The paper is organized as follows.  In Section 2, we review differential operators for Siegel modular forms and we sharpen a result by Maass on the Fourier expansions 
of Siegel modular forms that are annihilated by $\Omega_{\frac{1}{2}, k-\frac{1}{2}}$.  In Section 3, we discuss harmonic skew-Maass-Jacobi forms.  In Section 4, we explore Kohnen's limit process and we prove Theorem~\ref{main}.

\section{Differential operators for Siegel modular forms}
\label{sec:differentialoperators}

Maass \cite{Maass-diff} (see also \cite{Maass}) introduces differential operators for Siegel modular forms of degree $n$.  In this paper, we focus on real-analytic Siegel modular forms of degree $2$ and we only review the relevant results of Maass.  Let us start with some standard notation.  Let $\Mat{2}{\C}$ be the set of $2 \times 2$ matrices with entries in $\C$
and let $I_2\in\Mat{2}{\C}$ be the identity matrix. If $A\in\Mat{2}{\C}$, then $\tr(A)$ denotes the trace of $A$.  Moreover, let $\mbox{Sp}_{2}(\R)$ be the symplectic group of degree $2$ and let $Z=\left(\begin{smallmatrix} \tau & z \\ z & \tau'\end{smallmatrix}\right)=X+iY\in\HH_2$ be a typical variable.  As usual, if $M=\left(\begin{smallmatrix}A & B\\ C & D\end{smallmatrix}\right)\in\mbox{Sp}_{2}(\R)$ and $Z\in\HH_2$, then we set
$$
M\circ Z:=(AZ+B)(CZ+D)^{-1}.
$$
\noindent
Furthermore, for functions $G:\HH_2\rightarrow\C$ and for fixed $\alpha$, $\beta\in\C$ such that $\alpha-\beta\in\Z$, 
we define the slash operator 
\begin{equation}
\label{slash}
\left(G\,\big|_{(\alpha,\,\beta)}\,M\right)(Z):=\det(CZ+D)^{-\alpha}\det(C\overline{Z}+D)^{-\beta}\,G(M\circ Z)
\end{equation}
\noindent
for all $M=\left(\begin{smallmatrix}A & B\\ C & D\end{smallmatrix}\right)\in\mbox{Sp}_{2}(\R)$. 

\subsection{Casimir operators}

It is well known that the center of the universal enveloping algebra of $\mbox{Sp}_{2}(\R)$ is generated by $2$ elements, the {\it Casimir elements}.  
Their images under the action in (\ref{slash}) yield a quadratic and a quartic Casimir operator, which generate the $\C$-algebra of invariant differential operators with respect 
to the action in (\ref{slash}).  Maass (see $\S 8$ in \cite{Maass}) determines this algebra and we now introduce some more notation to give the explicit formulas of these invariant differential operators in Maass \cite{Maass}.  Let
\begin{gather*}
  \partial_Z
:=
  \left(\begin{matrix} \partial_{\tau} & \frac{1}{2}\partial_z \\
                       \frac{1}{2}\partial_z & \partial_{\tau'}
  \end{matrix}\right)
\quad \text{and} \quad
  \partial_{\overline{Z}}
:=
  \left(\begin{matrix} \partial_{\overline{\tau}} & \frac{1}{2}\partial_{\overline{z}} \\
                       \frac{1}{2}\partial_{\overline{z}} & \partial_{\overline{\tau'}}
  \end{matrix}\right)
\text{,}
\end{gather*}
where $\partial_{w}:=\frac{\partial}{\partial w}=\frac12\left(\frac{\partial}{\partial a}-i\frac{\partial}{\partial b}\right)$ and $\partial_{\overline{w}}:=\frac{\partial}{\partial \overline{w}}=\frac12\left(\frac{\partial}{\partial a}+i\frac{\partial}{\partial b}\right)$ for any complex variable $w=a+ib$.  Set
\begin{gather*}
  K_{\alpha}
:=
  \alpha I_2+(Z-\overline{Z})\partial_Z
\text{,} \qquad
  \Lambda_{\beta}
:=
  -\beta I_2+(Z-\overline{Z})\partial_{\overline{Z}}
\text{,}
\end{gather*}
and
\begin{align}
\label{Laplace}
  \Omega_{\alpha, \beta} 
&:=
   \Lambda_{\beta - \tfrac{3}{2}} K_{\alpha} + \alpha (\beta - \tfrac{3}{2})I_2
\\
\nonumber
&\hphantom{:}=
  -4 Y \trans (Y \partial_{\overline{Z}}) \partial_Z - 2 i \beta Y \partial_Z + 2 i \alpha Y \partial_{\overline Z}
\text{.}
\end{align}
Finally, if $A_{\alpha,\beta}^{(1)} := \Omega_{\alpha, \beta} - \alpha (\beta - \frac{3}{2} ) I_2$, then
\begin{align*}
  H_1^{(\alpha,\beta)}
&:= \tr\bigl(A_{\alpha,\beta}^{(1)}\bigr)
\qquad \text{and}
\\
  H_2^{(\alpha,\beta)}
&:=
     \tr\bigl(A_{\alpha,\beta}^{(1)}\, A_{\alpha,\beta}^{(1)}\bigr)
  - \tr\bigl(\Lambda_{\beta}\, A_{\alpha,\beta}^{(1)}\bigr)
  + \tfrac{1}{2} \tr\bigl(\Lambda_{\beta}\bigr) \tr\bigl(A_{\alpha,\beta}^{(1)}\bigr)
\end{align*}
are two Casimir operators that generate the $\C$-algebra of invariant differential operators with respect 
to the action in (\ref{slash}).
%
%

\vspace{1ex}

\begin{remark}
Nakajima \cite{Naka-Japan82} (apparently unaware of the Theorem on p. 116 of \cite{Maass}) gives the two invariant differential operators with respect to the action in (\ref{slash}) when $\alpha=\beta=0$.  The quadratic operator in \cite{Naka-Japan82} coincides (up to a constant factor) with $H_1^{(0,0)}$, but unfortunately, the quartic operator in \cite{Naka-Japan82} is incorrect.  This was verified with Singular and Plural~\cite{singular, plural} and the computer code is posted on the author's homepages.  Note also that the Fourier series expansions of Siegel-Maass forms in Niwa \cite{Niwa-Nagoya} rely on the differential operators of \cite{Naka-Japan82}. 
\end{remark}

\vspace{1ex}

Of particular interest is the operator
\begin{equation}
\label{factoring}
\mathcal{C}^{(\alpha, \beta)} := N_{\beta - 1} M_{\alpha}
\text{,}
\end{equation}
where $M_{\alpha}$ and $N_{\beta}$ are the raising and lowering operators, respectively, of Maass \cite{Maass-diff}.  Recall that
\begin{align*}
M_{\alpha}:=\alpha (\alpha - \tfrac{1}{2}) + (\alpha - \tfrac{1}{2})
    \bigl((\tau - \overline{\tau}) \partial_\tau+
          (z - \overline{z}) \partial_z+
          (\tau' - \overline{\tau'}) \partial_{\tau'}\bigr)+ \det (Z - \overline{Z}) (\partial_\tau \partial_{\tau'} - \tfrac{1}{4} \partial_z^2)
\end{align*}
and $N_{\beta} := \mathfrak{i} M_{\beta} \mathfrak{i}$ with $\mathfrak{i}(G)(Z):=G(-\overline{Z})$ for any $G:\HH_2\rightarrow\C$.  A direct computation shows that
\begin{gather*}
  \mathcal{C}^{(\alpha, \beta)}
=
     H_2^{(\alpha, \beta)}
   + \bigl(H_1^{(\alpha, \beta)}\bigr)^2
   + \tfrac{1}{2} ( 1 + \alpha - \beta) H_1^{(\alpha, \beta)}
\text{.}
\end{gather*}

Theorem~\ref{main} connects skew-holomorphic Jacobi forms and more generally, harmonic skew-Maass-Jacobi forms (see Section \ref{sec:skew-Jacobi}) with real-analytic Siegel modular forms of type $(\frac{1}{2}, k-\frac{1}{2})$.  
Therefore, we focus on the case $\alpha=\frac{1}{2}$ and in this case we write $H_1:=H_1^{(\frac{1}{2},k-\frac{1}{2})}$, $H_2:=H_2^{(\frac{1}{2},k-\frac{1}{2})}$, 
and $\mathcal{C} := \mathcal{C}^{(\frac{1}{2}, k - \frac{1}{2})}$.  Note that $\mathcal{C}=\xi_{3-k,0}^{(2)}\,\xi_{\frac{1}{2},k-\frac{1}{2}}^{(2)}$, where
\begin{gather}
\label{Siegel-xi}
\xi_{k, 0}^{(2)}:=\det(Y)^{k-\frac{3}{2}}\,N_0\quad \text{and} \quad \xi_{\frac{1}{2}, k-\frac{1}{2}}^{(2)}:=\det(Y)^{k-\frac{3}{2}}\,M_\frac{1}{2}
\end{gather}
are higher dimensional generalizations of Bruinier's and Funke's \cite{B-F-Duke04} operator $\xi_k$.  

In the spirit of Borel \cite{Borel}, it is natural to consider real-analytic Siegel modular forms that are eigenfunctions of the two Casimir operators $H_1$ and $H_2$.  
Without further restrictions it seems quite hopeless to explicitly describe the Fourier expansions of such forms.  Nevertheless, we can use the Laplace operator $\Omega_{\frac{1}{2}, k-\frac{1}{2}}$ in (\ref{Laplace}) to define a subspace of such real-analytic Siegel modular forms, which allows us to explicitly determine their Fourier coefficients.

\subsection{Siegel-Maass forms}
Imamo\=glu and the third author \cite{I-R-RIMS10} consider Siegel-Maass forms of degree $n$ that are annihilated by the Maass operator $M_{\frac{n-1}{2}}$.  
If $n=2$, then such forms are in the kernel of $\xi_{\frac{1}{2}, k-\frac{1}{2}}^{(2)}$ in (\ref{Siegel-xi}) and hence (in light of the degree $1$ case in \cite{B-F-Duke04}), 
they are analogs of ``holomorphic'' Siegel-Maass forms.  However, there are also Siegel-Maass forms that play a key role in the proof of Theorem~\ref{main} 
(see Theorem~\ref{thm:eisensteinlimit}), and which are not in the kernel of $\xi_{\frac{1}{2}, k-\frac{1}{2}}^{(2)}$ (see Proposition \ref{xi of Poincare series}).  
Our following definition includes such examples of Siegel-Maass forms.

\vspace{1ex}

\begin{defn}
\label{HSM}
A harmonic Siegel-Maass form of weight $k$ on $\Gamma:=\Sp{2}(\Z)$ is a real-analytic $F:\HH_2\rightarrow \C$ satisfying the following conditions:

\vspace{1ex}

(1) For all $M\in\Gamma$, $F\,|_{(\frac{1}{2},\,k-\frac{1}{2})}\,M=F$.

\vspace{1ex}

(2) We have that $\Omega_{\frac{1}{2}, k-\frac{1}{2}}(F)=0$.

\vspace{1ex}

(3) We have that $|F(Z)|\leq C \, \tr (Y)^N$ for some $C>0$ and $N\in\N$. 

\vspace{1ex}

Let $\widehat{\mathbb{M}}_k$ denote the space of such harmonic Siegel-Maass forms of weight $k$.
\end{defn}


\newpage

\begin{remarks}
\ 
\begin{enumerate}
\item
Maass \cite{Maass-diff} essentially shows that if $G : \HH_2 \rightarrow \C$ satisfies $\Omega_{\alpha, \beta}(G) = 0$, then $G$ is annihilated by all invariant differential operators with no constant term for the slash action $|_{(\alpha, \beta)}$ (for a representation theoretic proof of this fact see \cite{Raum-thesis}).  Hence $\widehat{\mathbb{M}}_k$ is a subspace of the space of real-analytic Siegel modular forms that are eigenfunctions of both Casimir operators $H_1$ and $H_2$.

\vspace{1ex}

\item
One cannot use the Koecher principle to remove condition (3), since there are harmonic Siegel-Maass forms that have singularities at the Satake boundary of the Siegel upper half space (see Proposition \ref{Poincare series is harmonic} in the case that $k<0$).

\vspace{1ex}

\item
Note that holomorphic Siegel modular forms of weight $k$ are annihilated by the matrix-valued Laplace operator $\Omega_{k,0}$.

\item
The space $\widehat{\mathbb{M}}_k$ is invariant under the action of the Hecke operators (for details on Hecke operators, 
see Chapter IV of \cite{Frei_1}): If $F\in\widehat{\mathbb{M}}_k$ and $T$ is a Hecke operator, then the definition of the Hecke operator implies that $F\,| \,T$ satisfies (1) and (3) of Definition \ref{HSM}.  The covariance property of $\Omega_{\frac{1}{2},k-\frac{1}{2}}$ (see $\S8$ of \cite{Maass}) shows that $F\,| \,T$ satisfies (2) of Definition \ref{HSM}.
\end{enumerate}
\end{remarks}

Examples of harmonic Siegel-Maass forms can be constructed via the Poincaré-Eisenstein series
\begin{equation}
\label{Poincare-Eisenstein series}
P_{k,s}(Z):=\sum_{M \in \Gamma_{\infty} \backslash \Gamma} \left((\det Y)^s |_{(\frac{1}{2},\,k -\frac{1}{2})} \,M\right)(Z),
\end{equation}
where $\Gamma_{\infty}:=\left\{\left(\begin{smallmatrix} A & B \\ C & D\end{smallmatrix}\right)\in\Gamma\,|\,C=0\right\}$.

\vspace{1ex}

\begin{remark}
Note that $P_{k,s}=(\det Y)^sE_{s+\frac{1}{2},s+k-\frac{1}{2}}$, where $E_{\alpha,\beta}$ is Maass' \cite{Maass-diff, Maass} nonholomorphic Eisenstein series.  We find that $P_{k,s}$ converges absolutely if $2 \mathrm{Re} (s)+k>3$.  Maass \cite{Maass-diff} also showed that $E_{\alpha,\beta}$ is in the kernel of $\Omega_{\alpha,\beta}$ (provided that $\alpha-\beta\in 2\Z$ and $\mathrm{Re} (\alpha+\beta)>3$). 
\end{remark}

\vspace{1ex}

\begin{prop}
\label{Poincare series is harmonic}
If $s=0$ ($k>3$) or $s=\frac{3}{2}-k$ ($k<0$), then $P_{k,s}\in\widehat{\mathbb{M}}_k$.
\end{prop}

\begin{proof}
A direct computation shows that
$$
\Omega_{\frac{1}{2},k-\frac{1}{2}}\left((\det Y)^s\right)=-s\left(s-\left(\tfrac{3}{2}-k\right)\right)(\det Y)^sI_2
$$
and the covariance of $\Omega_{\frac{1}{2},k-\frac{1}{2}}$ proves that $\Omega_{\frac{1}{2},k-\frac{1}{2}}(P_{k,s})=0$ for $s=0$ and $s=\frac{3}{2}-k$.  Finally, (3) of Definition \ref{HSM} is satisfied for $E_{s+\frac{1}{2},s+k-\frac{1}{2}}$ and hence also for $P_{k,s}$.
\end{proof}

\vspace{1ex}

\begin{remark}
Proposition \ref{Poincare series is harmonic} in combination with Remarks $(1)$ after Definition \ref{HSM} implies that $\mathcal{C} (P_{k,s}) = 0$ for $s = 0$ and $s = \tfrac{3}{2} - k$.  
It is easy to verify that $\mathcal{C} (P_{k,s}) = 0$ also for $s=-\frac{1}{2}$ ($k>4$) and $s=2-k$ ($k<1$).
\end{remark}

Our final result in this subsection gives the image of $P_{k,s}$ under $\xi_{\frac{1}{2}, k-\frac{1}{2}}^{(2)}$ if $s=0$ and $s=\tfrac{3}{2}-k$.  In particular, if $s=\tfrac{3}{2}-k$, then $P_{k,s}$ is not a Siegel-Maass form as in \cite{I-R-RIMS10}.

\vspace{1ex}

\begin{prop}
\label{xi of Poincare series}
If $s=0$, then $P_{k,s}$ is already annihilated by $\xi_{\frac{1}{2}, k-\frac{1}{2}}^{(2)}$.  If $s=\frac{3}{2}-k$, then
\begin{equation*}
\xi_{\frac{1}{2}, k-\frac{1}{2}}^{(2)}(P_{k,s}) =\left(k-\tfrac{3}{2}\right)\left(k-2\right)E_{3-k},
\end{equation*}
where $E_{3-k}$ is the usual holomorphic Siegel-Eisenstein series of weight $3-k$.
\end{prop}
\begin{proof}
A direct computation shows that
$$
M_{\frac{1}{2}}\bigl((\det Y)^s\bigr)=s(s+\tfrac{1}{2})(\det Y)^s
$$
which proves the claim.
\end{proof}

\subsection{Fourier series expansions}
\label{ssec:fourierexpansions}

Maass \cite{Maass-diff} determines the Fourier series expansions of functions that are in the kernel of $\Omega_{\alpha, \beta}$.  We will first recall Maass' result (where we have slightly changed the notation of some variables to avoid confusion with our earlier notation), and then we will improve it in the case of harmonic Siegel-Maass forms.

\vspace{1ex}

\begin{thm}[\cite{Maass-diff}]
\label{thm:laplacefourierexpansion}
\label{thm:siegelmodularforms:harmonicfourierexpansions} 
Let $G(Z) = a(Y, T)\, e^{i \,\tr(TX)}$, where $T$ is a real symmetric $2\times 2$-matrix, and suppose 
$\Omega_{\alpha,\beta} (G) = 0$ where $\alpha+\beta\not= 1, \, \frac{3}{2}, \, 2$.  Write
$$
Y=\sqrt{\det Y}\left(\begin{matrix} (\text{\tt x}^2+\text{\tt y}^2)\text{\tt y}^{-1} & \text{\tt x}\text{\tt y}^{-1}\\ \text{\tt x}\text{\tt y}^{-1} & \text{\tt y}^{-1}\end{matrix}\right)
$$
and
$$
\text{\tt u}:=\tr(YT), \hspace{3ex} \text{\tt v}:=\big(\tr(YT)\big)^2-4\det (YT).
$$
Then $a(Y, T)$ is given as follows:
\begin{enumerate}[(a)]
\item If $T = 0$, then
\begin{equation}
\label{rank 0}
a(Y,0)=\phi(\text{\tt x},\text{\tt y})\det Y^{\frac{1}{2}(1-\alpha-\beta)}+c_1 \det Y^{\frac{3}{2}-\alpha-\beta}+c_2,
\end{equation}
where $c_1, c_2\in\C$ and $\phi(\text{\tt x},\text{\tt y})$ is an arbitrary solution (analytic for $\text{\tt y}>0$) of the wave equation
\begin{equation*}
\text{\tt y}^2(\phi_{\text{\tt x}\text{\tt x}}+\phi_{\text{\tt y}\text{\tt y}})-(\alpha+\beta-1)(\alpha+\beta-2)\phi=0.
\end{equation*}

\vspace{1ex}

\item If $\rank (T) = 1$, $T\geq 0$, then
\begin{equation}
\label{rank 1}
a(Y,T)=\phi(\text{\tt u})\det Y^{\frac{3}{2}-\alpha-\beta}+\psi(\text{\tt u}),
\end{equation}
where $\phi$ and $\psi$ are confluent hypergeometric functions that satisfy the following differential equations
\begin{equation*}
\begin{split}
\text{\tt u}\phi'' & +(3-\alpha-\beta)\phi'+(\alpha-\beta-\text{\tt u})\phi=0\\
\text{\tt u}\psi'' & +(\alpha+\beta)\psi'+(\alpha-\beta-\text{\tt u})\psi=0.\\
\end{split}
\end{equation*}
In particular, there are four linear independent solutions $a(Y,T)$ in this case.

\vspace{1ex}

\item \label{it:siegelmodularforms:realanalyticfourierexpansion_rkT2_Tg0} 
If $\rank (T) = 2$, $T>0$, then
\begin{equation}
\label{rank 2}
a(Y,T)=\sum_{n=0}^{\infty}g_n(\text{\tt u})\text{\tt v}^n \hspace{3ex} (|\text{\tt v}|<\text{\tt u}^2),
\end{equation}
where the functions $g_n(\text{\tt u})$ are recursively defined by
\begin{equation*}
4(n+1)^2\text{\tt u}g_{n+1}+\text{\tt u}g_n''+2(2n+\alpha+\beta)g_n'+(2(\alpha-\beta)-\text{\tt u})g_n=0
\end{equation*}
and
\begin{gather*}
  g_0(\text{\tt u})=\text{\tt u}^{1-\alpha-\beta} \, \psi(\text{\tt u}),
 \quad\text{with}\quad
 \psi'(\text{\tt u})=\frac{1}{\text{\tt u}}\phi(\text{\tt u})
 \text{,}\quad\text{and}
\\
  \phi''=\left(1+\frac{2(\beta-\alpha)}{\text{\tt u}}+\frac{(\alpha+\beta-1)(\alpha+\beta-2)}{\text{\tt u}^2}\right)\phi
\text{.}
\end{gather*}
In particular, there are three linear independent solutions $a(Y,T)$ in this case.

\vspace{1ex}

\item If $\rank (T) = 2$, $T \text{ indefinite}$, then
\begin{gather*}
a(Y, T) = \sum_{n = 0}^\infty h_n (\text{\tt v}) \text{\tt u}^n
\quad (\text{\tt u}^2 < \text{\tt v})
\text{,}
\end{gather*}
where the functions $h_n (\text{\tt v})$ are recursively defined by
\begin{gather*}
(n + 2) (n + 1) h_{n + 2} + 4 \text{\tt v} h''_n + 4 (\alpha + \beta + n)h'_n - h_n = 0
\end{gather*}
and
\begin{align*}
  (\alpha - \beta) h_1
&=
   8 \text{\tt v}^2 h'''_0 + 4(2 + 3 \alpha + 3 \beta)\text{\tt v} h''_0
\\
  &\qquad+ (4 (\alpha + \beta)^2 + 2 (\alpha + \beta - 1) - 2 \text{\tt v}) h'_0 - (\alpha + \beta)h_0
\text{,}
\\
  (\beta - \alpha) h_0
&=
  2 \text{\tt v} h'_1 + (\alpha + \beta) h_1
\text{.}
\end{align*}
In particular, there are four linear independent solutions $a(Y,T)$ in this case.
\end{enumerate}
Finally, any solution for the data $\{\alpha, \beta, T\}$ is also a solution for the data $\{\beta,\alpha, -T\}$. 
\end{thm}

\vspace{1ex}

We now recall some standard special functions, which are needed to state our results in the next theorem and remarks.  
Let $M_{\nu,\mu}$ and $W_{\nu,\mu}$ be the usual $M$-Whittaker function and $W$-Whittaker function, respectively, 
which are solutions to the differential equation
\begin{equation*} 
 \frac{\partial^2}{\partial w^2}f(w)+
 \left(-\frac{1}{4}+\frac{\nu}{w}+\frac{\frac{1}{4}-\mu^2}{w^2}\right)f(w)=0.
\end{equation*}
For fixed $\nu$ and $\mu$ we have the following asymptotic behavior as $y\ra\infty$
\begin{equation}
\label{Whittaker growth}
M_{\nu,\mu}(y)\sim \frac{\Gamma(1+2\mu)}{\Gamma(\mu-\nu+\tfrac{1}{2})}y^{-\nu}e^{\frac{y}{2}} \hspace{3ex} \text{and} \hspace{3ex} W_{\nu,\mu}(y)\sim y^{\nu}e^{-\frac{y}{2}}\,,
\end{equation}
where $\Gamma(\cdot)$ is the Gamma-function.  As usual, let $\Gamma(a ,y):=\int_{y}^{\infty}e^{-w}w^{a-1}\ dw$ denote the incomplete Gamma-function.  If $y\ra\infty$, then
\begin{equation}
\label{incomplete Gamma growth}
\Gamma \left( a,y \right)\sim y^{a-1}e^{-y}\,. 
\end{equation}
Let ${}_p {\rm F}_q$ be the generalized hypergeometric series
\begin{gather*}
  {}_p {\rm F}_q (a_1,\ldots,a_p;b_1, \ldots, b_q; z) 
:=
  \sum_{n} \frac{(a_1)_n \cdots (a_p)_n}{(b_1)_n \cdots (b_q)_n} \frac{z^n}{n!}
\text{,}
\end{gather*}
where $(a)_n:=a(a+1)(a+2)...(a+n-1)$ is the Pochhammer symbol.
The asymptotic behavior of the generalized hypergeometric function is quite complicated (see for example $\S5.11$ of \cite{Luke-I}), and we only remark here that the generalized hypergeometric function growths rapidly for generic parameters. 

Our next theorem sharpens Theorem~\ref{thm:laplacefourierexpansion} in the case of harmonic Siegel-Maass forms.  Note that the exponentials of the Fourier series expansions in Theorem~\ref{thm:laplacefourierexpansion} and Theorem~\ref{thm:laplacesiegelfourierexpansion} differ by $2 \pi$.

\vspace{1ex}

\begin{thm}
\label{thm:laplacesiegelfourierexpansion}
Let $\displaystyle F(Z)=\sum_{T} a(Y,T)e^{2 \pi i\,\tr(TX)}\in\widehat{\mathbb{M}}_k$.  
As in Theorem~\ref{thm:laplacefourierexpansion} we write
$$
\text{\tt u}:=\tr(YT) \hspace{2ex} \text{and} \hspace{2ex} \text{\tt v}:=\big(\tr(YT)\big)^2-4\det (YT),
$$
and $c_1,c_2\in\C$ are always constants. Then $a(Y, T)$ is given as follows:
\begin{enumerate}[(a)]
\item If $T = 0$, then
\begin{equation}
\label{rank 0'}
a(Y,0)=c_1 \det Y^{\frac{3}{2}-k}+c_2,
\end{equation}
which is in the kernel of $\xi_{\frac{1}{2}, k-\frac{1}{2}}^{(2)}$ if and only if $c_1=0$.

\item
\label{eq:rank 1' positive}
If $\rank (T) = 1$, $T\geq 0$, then two of the four fundamental solutions of $(b)$ in Theorem~\ref{thm:laplacefourierexpansion} do not occur.  Any Fourier coefficient that occurs is of the form 
\begin{gather*}
  c_1\, \text{\tt u}^{k-2} \det Y^{\frac{3}{2} - k}\, e^{2 \pi \text{\tt u}\,} \Gamma(2 - k, 4 \pi \text{\tt u}) 
+ c_2\, \text{\tt u}^{-\frac{k}{2}} W_{\frac{1-k}{2}, \frac{k - 1}{2}} (4 \pi \text{\tt u})
\text{,}
\end{gather*}
and $\xi_{\frac{1}{2}, k-\frac{1}{2}}^{(2)}\left(a(Y,T)e^{2 \pi i\,\tr(TX)}\right)=0$ if and only if $c_1=0$.

\vspace{1ex}


\vspace{1ex}

\item
\label{it:Tge0_solution_occurring}
If $\rank (T)=2$, $T >0$, then two of the three fundamental solutions of $(c)$ in Theorem~\ref{thm:laplacefourierexpansion} do not occur.  Any Fourier coefficient that occurs is of the form
\begin{gather*}
c_1\; \sum_{n = 0}^\infty g_n(2 \pi \text{\tt u})\, (4 \pi^2 \text{\tt v})^n
\end{gather*}
with $g_n$ as in $(c)$ of Theorem~\ref{thm:laplacefourierexpansion} and
\begin{gather*}
g_0 (\text{\tt u}) = \text{\tt u}^{1-k}\int_{\text{\tt u}}^{\infty} {\tilde{\text{\tt u}}}^{-1} W_{1 - k, (\sgn k) (k - \frac{3}{2})} (2 \tilde{\text{\tt u}}) \;d\tilde{\text{\tt u}}
\text{,}
\end{gather*}
and $\xi_{\frac{1}{2}, k-\frac{1}{2}}^{(2)}\left(a(Y,T)e^{2 \pi i\,\tr(TX)}\right)=0$ if and only if $c_1=0$.

\vspace{1ex}



\item If $\rank (T)=2$, $T$ indefinite, then three of the four fundamental solutions of $(d)$ in Theorem~\ref{thm:laplacefourierexpansion} do not occur.
\end{enumerate}
\end{thm}

\vspace{1ex}

Before we prove Theorem~\ref{thm:laplacesiegelfourierexpansion} we give the additional solutions (computed with Mathematica) of Theorem~\ref{thm:laplacefourierexpansion} 
that do not occur in Theorem~\ref{thm:laplacesiegelfourierexpansion}. 

\vspace{1ex}

\begin{remarks}
\begin{enumerate}[(a)]
\setcounter{enumi}{1}

\item If $\rank (T) = 1$, $T\ge 0$, then the two additional fundamental solutions are:
\begin{gather*}
\text{\tt u}^{k-2} e^{2 \pi \text{\tt u}}
\quad \text{and} \quad
\begin{cases}
  e^{-2 \pi \text{\tt u}\,} {}_1 {\rm F}_1 (-\frac{1}{2} + k; k; 4 \pi \text{\tt u})
& \text{if }k > 3\text{,}
\\
  \text{\tt u}^{1 - k} e^{-2 \pi \text{\tt u}\,} {}_1 {\rm F}_1 (\frac{1}{2}; 2 - k; 4 \pi \text{\tt u})
& \text{if }k < 0\text{.}
\end{cases}
\end{gather*}
\vspace{1ex}


\item If $T > 0$, then the two additional solutions arise via
$$
g_0(\tt u)=\text{\tt u}^{1 - k}
$$ 
and
\begin{gather*}
g_0 (\text{\tt u}) = \text{\tt u}^{1-k}\int_{\text{\tt u}}^{\infty} {\tilde{\text{\tt u}}}^{-1} M_{1 - k, (\sgn k) (k - \frac{3}{2})} (2 \tilde{\text{\tt u}}) \;d\tilde{\text{\tt u}}
\text{.}
\end{gather*}



\vspace{1ex}



\item If $\rank (T) = 2$, $T \text{ indefinite}$, then the three additional fundamental solutions for $h_1$ are different when $k > 3$ and $k < 0$.  If $k > 3$ they are:
\begin{gather*}
  {}_1 {\rm F}_2\bigl(\tfrac{1}{2}; \tfrac{1 + k}{2}, 1 + \tfrac{k}{2}; \pi^2  \text{\tt v}\bigr)
\text{,}\qquad
  \text{\tt v}^{\frac{-k}{2}} {}_1 {\rm F}_2 \bigl(\tfrac{1 - k}{2}; \tfrac{1}{2}, 1 - \tfrac{k}{2}; \pi^2 \text{\tt v}\bigr)
\end{gather*}
and
\begin{gather*}
  (\pi^2 \text{\tt v})^{\tfrac{1-k}{2}} {}_1 {\rm F}_2\bigl(1 - \tfrac{k}{2}; \tfrac{3}{2}, \tfrac{3 - k}{2}; \pi^2 \text{\tt v}\bigr)
  -
  \frac{\big(1 - \frac{k}{2}\big))_{\frac{k - 1}{2}} \Gamma\big(\frac{3 - k}{2}\big)}
       {\big(\frac{3}{2}\big)_{\frac{k - 1}{2}} \big(\frac{k - 1}{2}\big)!}
  {}_1 {\rm F}_2\bigl(\tfrac{1}{2}; \tfrac{1 + k}{2}, 1 + \tfrac{k}{2}; \pi^2  \text{\tt v}\bigr)
\text{.}
\end{gather*}
Note that the second solution above is a Laurent polynomial in $\text{\tt v}^{-\tfrac{1}{2}}$.

If $k < 0$, then the three additional fundamental solutions for $h_1$ are given by
\begin{gather*}
\text{\tt v}^{\tfrac{1-k}{2}} {}_1 {\rm F}_2\bigl(1 - \tfrac{k}{2}; \tfrac{3}{2}, \tfrac{3 - k}{2}; \pi^2 \text{\tt v}\bigr)
\text{,} \qquad
\text{\tt v}^{\tfrac{3}{2} - k} {}_2 {\rm F}_3\bigl(1, 2 - k; \tfrac{5}{2} - k, 2 - \tfrac{k}{2}, \tfrac{5 - k}{2}; \pi^2 \text{\tt v}\bigr)
\quad \text{and}
\\
{}_1 {\rm F}_2\bigl(\tfrac{1}{2}; \tfrac{1 + k}{2}, 1 + \tfrac{k}{2}; \pi^2  \text{\tt v}\bigr)
-
\frac{\big( \frac{1}{2} \big)_{\frac{1 - k}{2}}
      \Gamma\big(\frac{k + 1}{2}\big)}
     {\big( 1 + \frac{k}{2} \big)_{\frac{1 - k}{2}}
      \big(\frac{1 - k}{2}\big)!}
\big(\pi^2\text{\tt v}\big)^{\tfrac{1-k}{2}} {}_1 {\rm F}_2\bigl(1 - \tfrac{k}{2}; \tfrac{3}{2}, \tfrac{3 - k}{2}; \pi^2 \text{\tt v}\bigr)
\text{.}
\end{gather*}

Individual generalized hypergeometric series may not be not defined for some $k$, but linear combinations of generalized hypergeometric series can be analytically continued for such $k$, and we always refer to their analytic continuations.
\end{enumerate}


\end{remarks}

\vspace{1ex}

\begin{proof}[\indent \bf Proof of Theorem~\ref{thm:laplacesiegelfourierexpansion}:]
We used Mathematica to find the explicit solutions for $a(Y,T)$ in Theorem~\ref{thm:laplacefourierexpansion}.  
It is easy to see that the functions in $(a)$, $(b)$, and $(c)$ of Theorem~\ref{thm:laplacesiegelfourierexpansion} and the functions in $(b)$ and $(c)$ of the remarks to Theorem~\ref{thm:laplacesiegelfourierexpansion} are indeed solutions, and also that the functions in $(a)$, $(b)$, and $(c)$ of Theorem~\ref{thm:laplacesiegelfourierexpansion} yield solutions that satisfy the growth condition $(3)$ of Definition \ref{HSM}. 
The case where $T$ is indefinite is more complicated.  We will first verify directly that the functions in $(d)$ of the remarks to Theorem~\ref{thm:laplacesiegelfourierexpansion} are solutions.  Then we will show that no linear combination of the solutions given in the remarks to Theorem~\ref{thm:laplacesiegelfourierexpansion} satisfies the growth condition $(3)$ of Definition \ref{HSM}.
Finally, we will show that for $T > 0$ and $k > 3$ any possible nontrivial solution is not in the kernel of $\xi_{\frac{1}{2}, k - \frac{1}{2}}^{(2)}$.

The computations are quite involved and where partially performed with the help of Sage~\cite{sage} and Singular~\cite{singular}.  The computer code is posted on the author's homepages.

\vspace{1ex}

Let $T$ be indefinite.  We confirm that the functions in $(d)$ of the remarks to Theorem~\ref{thm:laplacesiegelfourierexpansion} are solutions for $a(Y,T)$ in Theorem~\ref{thm:laplacefourierexpansion} 
by showing that for generic $k$ the following generalized hypergeometric series are the solutions for $h_1(\text{\tt v})$ in $(d)$ of Theorem~\ref{thm:laplacefourierexpansion}:
\begin{equation}
\label{eq:indefinitehypergeometric}
\begin{split}
&
{}_1 {\rm F}_2\bigl(\tfrac{1}{2}; \tfrac{1 + k}{2}, 1 + \tfrac{k}{2}; \tfrac{\text{\tt v}}{4}\bigr)
\text{,}\quad
\text{\tt v}^{-\tfrac{k}{2}} {}_1 {\rm F}_2\bigl(\tfrac{1 - k}{2}; \tfrac{1}{2}, 1 - \tfrac{k}{2}; \tfrac{\text{\tt v}}{4}\bigr)
\text{,}
\\&
\text{\tt v}^{\tfrac{1-k}{2}} {}_1 {\rm F}_2\bigl(1 - \tfrac{k}{2}; \tfrac{3}{2}, \tfrac{3 - k}{2}; \tfrac{\text{\tt v}}{4}\bigr),
\quad \text{and} \quad
\text{\tt v}^{\tfrac{3}{2} - k} {}_2 {\rm F}_3\bigl(1, 2 - k; \tfrac{5}{2} - k, 2 - \tfrac{k}{2}, \tfrac{5 - k}{2}; \tfrac{\text{\tt v}}{4}\bigr)
\text{.}\\
\end{split}
\end{equation}

We will need the following Lemma on generalized hypergeometric series with parameters ${\bf a} = a_1,\ldots, a_p$ and ${\bf b} = b_1,\ldots,b_q$, 
where $a_i,b_j\in\C[k]$.
\begin{lemma}
\label{la:hypergeometricsolution}
Suppose $\mathcal{D}$ is an order $D$ linear differential operator on smooth functions of $\text{\tt v}$.  Assume that $\mathcal{D}$ has coefficients in $\C[\text{\tt v}, k]$, and that these coefficients have maximal degree $m_{\text{\tt v}}$ in $\text{\tt v}$.  If $l\in\Z$ and all $b_j$'s are either positive or nonintegral, then
\begin{gather*}
  \mathcal{D}\, \text{\tt v}^l {}_p {\rm F}_q ({\bf a}; {\bf b}; \text{\tt v}) = 0
\end{gather*}
if and only if the $t$-th coefficients ($l-D\leq t \leq l+D+m_{\text{\tt v}}$) of $\mathcal{D}\, \text{\tt v}^l {}_p {\rm F}_q ({\bf a}; {\bf b}; \text{\tt v})$ vanish as functions of $k$.
\end{lemma}
\begin{proof}
It suffices to prove that
\begin{gather*}
  \mathcal{D} \, \text{\tt v}^l {}_p {\rm F}_q ({\bf a}; {\bf b}; \text{\tt v})
=
  {\tt v}^{l - D} \bigl({}_p {\rm F}_q ({\bf a} + D; {\bf b} + D; \text{\tt v}) p_1
                       + p_2 \bigr)
\end{gather*}
for some $p_1, p_2\in\C(k)[\tt v]$ of degree at most $2 D + m_{\text{\tt v}}$.  Without loss of generality let $\mathcal{D}=\partial_{\tt v}^i$ with $i \in \{0,\ldots,D\}$ and, in particular, $m_{\text{\tt v}} = 0$.

We proceed by mathematical induction on $D$.  The case $D = 0$ is clear. 
Suppose $\mathcal{D} = c_1 \partial_{\text{\tt v}} \tilde{\mathcal{D}} + c_2$ for some constants $c_1, c_2$ and an order $D - 1$ operator $\tilde{\mathcal{D}}$.  By induction hypothesis we have
\begin{gather*}
  \tilde{\mathcal{D}} \, \text{\tt v}^l {}_p {\rm F}_q ({\bf a}; {\bf b}; \text{\tt v})
= 
  {\tt v}^{l - D + 1} \bigl({}_p {\rm F}_q ({\bf a} + D - 1; {\bf b} + D - 1; \text{\tt v}) \tilde{p}_1
                          + \tilde{p}_2 \bigr)
\text{,}
\end{gather*}
where $\tilde{p}_1, \tilde{p}_2$ have maximal degree $2 D - 2$.  The definition of the generalized hypergeometric functions implies the relations
$$
  \text{\tt v}^l{}_p {\rm F}_q ({\bf a}; {\bf b}; \text{\tt v})
= 
  \text{\tt v}^{l - 1} \Big( \prod_i a_i \prod_j b_j^{-1} \Big) (\text{\tt v} + \text{\tt v}^2 \, {}_p {\rm F}_q ({\bf a} + 1; {\bf b} + 1; \text{\tt v}))
$$
and
$$
\partial_{\text{\tt v}}\, \text{\tt v}^l {}_p {\rm F}_q ({\bf a}; {\bf b}; \text{\tt v})
=
\text{\tt v}^{l - 1} \bigg(\Big( \prod_i a_i \prod_j b_j^{-1}\Big) \text{\tt v}\, {}_p {\rm F}_q ({\bf a} + 1; {\bf b} + 1; \text{\tt v})+l\,{}_p {\rm F}_q ({\bf a}; {\bf b}; \text{\tt v})\bigg)
\text{,}
$$
which yield the claim.
\end{proof}

Lemma \ref{la:hypergeometricsolution} allows us to reduce the proof to a computation of finitely many coefficients in a series expansion with respect to $\text{\tt v}$.  Note that the defining differential equations for $h_0$ and $h_1$ in $(d)$ of Theorem~\ref{thm:laplacefourierexpansion} imply
\begin{equation}
\label{application of Lemma}
\begin{split}
 0=& - 16 {\tt v}^3 \, \partial_{\tt v}^4 h_1
 - 32 (k + 2) {\tt v}^2 \, \partial_{\tt v}^3 h_1
 - 4 (5 k^2 + 15 k + 7 - {\tt v}) {\tt v} \, \partial_{\tt v}^2 h_1
\\&
 - 2 (2 k^3 + 5 k^2 + k - 2 - 2 k {\tt v} - 2 {\tt v}) \, \partial_{\tt v} h_1
 + (2 k - 1) \, h_1\,.\\
\end{split}
\end{equation}
By Lemma \ref{la:hypergeometricsolution}, we only need to verify that $11$ ($D=4, m_{\text{\tt v}} = 3$) coefficients of the left hand side of (\ref{application of Lemma}) vanish if $h_1$ is any of the generalized hypergeometric functions in (\ref{eq:indefinitehypergeometric}). With the help of Sage \cite{sage} we found that these coefficients vanish indeed, which proves that the functions in (\ref{eq:indefinitehypergeometric}) are solutions for $h_1(\text{\tt v})$ in $(d)$ of Theorem~\ref{thm:laplacefourierexpansion}.


\vspace{1ex}

Now we show that no linear combination of the functions in $(b)$, $(c)$, and $(d)$ of the remarks to Theorem~\ref{thm:laplacesiegelfourierexpansion} occurs as a solution for $a(Y,T)$.  First, we also have to exclude the solution to the wave equation in (\ref{rank 0}).

Consider $T = 0$.  If $U \in \GL{2}(\Z)$, then $F(\trans{U}ZU)=(\det U)^k F(Z) = \pm F(Z)$.  Hence $a(Y,0)= \pm a(\trans{U}YU, 0)$, 
and we used Sage \cite{sage} and Singular \cite{singular} to show that the solution in (\ref{rank 0}) reduces to (\ref{rank 0'}).

\vspace{1ex}

For the remaining cases we will analyze the growth of Fourier coefficients.  Note that if $F(Z)=\sum_{T} a(Y,T)e^{2 \pi i\,\tr(TX)}\in\widehat{\mathbb{M}}_k$, 
then
\begin{gather*}
  a(Y,T)
=
  \int_{\RR^3} F(Z) e^{- 2\pi i\tr(TX)} \; dX
\text{,}
\end{gather*}
and condition $(3)$ of Definition \ref{HSM} implies that $a(Y,T)$ does not grow rapidly.


Consider $\rank(T)=1$.  The asymptotic behavior of the exponential function and generalized hypergeometric series show that no linear combination of the functions in $(b)$ of the remarks to Theorem~\ref{thm:laplacesiegelfourierexpansion} occurs as a solution for $a(Y,T)$.

\vspace{1ex}

To treat the case $T > 0$ we will need the following lemma (see also~\cite{Raum-thesis}), which uses the valuation of a Laurent polynomial in $\ttu$ normalized by
\begin{gather*}
  {\rm val}_\ttu p
:=
  \max\{ l \in \ZZ \,:\, \ttu^{-l} p \in \CC[\ttu] \}
\text{.}
\end{gather*}


\begin{lemma}
\label{la:siegelmodularforms:coefficients_of_power_series_with_recursion}
Suppose that a sequence of Laurent polynomials $l_n$ in $\ttu$ satisfies a recursion of the form
\begin{gather*}
  l_{n + 1}
=
  \sum_{d = 0}^D p_{n,d}\, l_n^{(d)}
\text{,}
\end{gather*}
where $D\geq 0$, $l_n^{(d)}$ is the $d$-th derivative of $l_n$, and $p_{n,d}$ are Laurent polynomials in $\ttu$ with $\deg_\ttu p_{n,0} = 0$ and $\deg_\ttu p_{n,d} < d$ for $d \ne 0$.  Assume that the valuation of all $p_{n,d}$ 
is uniformly bounded, and let $V$ be a lower bound on ${\rm val}_\ttu (p_{n,d}) - d$.  Suppose that $(n |V|)^d \, p_{n,d}$ has uniformly bounded coefficients as $n \rightarrow \infty$.  If the leading coefficients of $l_0$ and $p_{n,0}$ are positive, then there is a constant $\kappa$ such that the series
\begin{gather}
\label{eq:siegelmodularforms:power_series_with_recursion}
  \sum_{n = 0}^\infty l_n \cdot \Big(\frac{\ttu}{\kappa}\Big)^n
\end{gather}
is well-defined as a formal Laurent series, and such that (\ref{eq:siegelmodularforms:power_series_with_recursion}) has bounded coefficients.

If, in addition,
\begin{multline}
\label{eq:siegelmodularforms:asymptotic_recursion_contribution}
  {{n - i + \#\{(d,j) : (p_{{\td n},d})_j \ne 0 \text{ for some }{\td n}\} - 1} \choose n - i}
\\[4pt]
  \cdot
  \big( |V| + |D_l| + (D_l - V_l) \big)^{n - i} 
  \bigg( \prod_{n' = j + 1}^{n + j} (p_{n', 0})_0 \bigg)
  \Big( \max_{(d,j) \ne (0,0)}  h_{d,j} \Big)^{n - i}
\end{multline}
is bounded for $n \ge 1$, $j \ge \deg_\ttu l_0$, and $0 \le i \le \min\{n, \deg_\ttu l_0 - {\rm val}_\ttu l_0\}$, where the first factor of (\ref{eq:siegelmodularforms:asymptotic_recursion_contribution}) is the usual binomial coefficient and
\begin{gather*}
  h_{d,j}
:=
  \Bigg(
  \sum_{n' = 0}^{n + j} \frac{(n' + 1)^d \big| (p_{n',d})_j \big|}
                          {(p_{n',0})_0}
  \Bigg)^{\frac{1}{d - j}}
\text{,}          
\end{gather*}
then $\kappa$ can be chosen such that all coefficients of $\ttu^j$ with $j > \deg_\ttu l_0$ in~\eqref{eq:siegelmodularforms:power_series_with_recursion} are positive. In particular, in this case (\ref{eq:siegelmodularforms:power_series_with_recursion}) grows rapidly as $\ttu \rightarrow \infty$.

%
\end{lemma}
\begin{proof}
We prove the first part of Lemma \ref{la:siegelmodularforms:coefficients_of_power_series_with_recursion} and for brevity we only sketch the quite technical proof of the second part (for more details see~\cite{Raum-thesis}).

Write $(p)_j$ for the $j^\thup$ coefficient of a polynomial $p$.  Set $D_l := \deg_\ttu l_0$.  The assumptions imply that $\deg_\ttu l_n\leq D_l$ and that the leading coefficient of $l_n$ is positive. 
Let $b \ge 1$ be a bound of $\big( |{\rm val}_\ttu(l_0)| + \big(D_l - {\rm val}_\ttu (l_0) \big) + (n + 1) |V| \big)^d \, \sum_{d, j} \big|(p_{n,d})_j\big|$ for all $n$.  The valuation of $l_n$ is bounded from below by ${\rm val}_\ttu(l_0) + nV$.  Let $B$ be a bound of the absolute values of the coefficients of $l_0$.  Then an induction establishes that the absolute value of the coefficients of $l_n$ is less than $B \, b^n$.  Choosing $\kappa:= 2 b$ shows that the Laurent series (\ref{eq:siegelmodularforms:power_series_with_recursion}) is well-defined.  Moreover, the $j^\thup$ coefficient of (\ref{eq:siegelmodularforms:power_series_with_recursion}) is bounded by $2B$, since
\begin{align*}
\quad
&  
  \bigg| \sum_{n = 0}^\infty \frac{(l_n)_{j - n}}{\kappa^n} \bigg|
\le
  \sum_{n = 0}^\infty \frac{B b^n}{\kappa^n}
\le
  2 B
\text{.}
\end{align*}

To prove the second part, we will need to determine how the coefficients of $l_n$ ($n>0$) depend on those of $l_0$.  We decompose $p_{n,d}$ into monomials, and with a slight abuse of notation we write (the non-commutative product)
\begin{gather*}
  (l_n)_{j}
=
  \Bigg( \bigg( \prod_{n' < n} \sum_{d_{n'},\, j_{n'}'}   (p_{n',d_{n'}})_{j_{n'}'} \, \ttu^{j_{n'}'} \, \partial_\ttu^{d_{n'}} \bigg) \sum_i (l_0)_i \, \ttu^i \Bigg)_j
\text{,}
\end{gather*}
where one first differentiates with respect to $\partial_\ttu^{d_{0}}$, then $\partial_\ttu^{d_{1}}$, etc.  For each contribution, \mbox{$\sum_{n'} \big( d_{n'} - j_{n'}' \big)$} is bounded by $i - j$.  If $(p_{n',d_{n'}})_{j'_{n'}} \ne 0$ and $\big( d_{n'},\, j_{n'}'\big) \ne (0,0)$, then $d_{n'} - j_{n'}' > 0$, and we find that the  $j^\thup$ coefficient of $l_n$ only depends on $(l_0)_i$ by means of ``products''
\begin{gather*} 
  (p_{n,d_{n}})_{j_{n}'} \, \ttu^{j_{n}'} \, \partial_\ttu^{d_{n}} \, \cdots \,
  (p_{n',d_{n'}})_{j_{n'}'} \, \ttu^{j_{n'}'} \, \partial_\ttu^{d_{n'}} \, \cdots \,
 (p_{0,d_{0}})_{j_{0}'} \, \ttu^{j_{0}'} \, \partial_\ttu^{d_{0}} 
\;
 (l_0)_i \, \ttu^i
\text{,}
\end{gather*}
with at most $i - j$ pairs $\big( d_{n'},\, j_{n'}'\big)$ different from $(0,0)$.  The proof proceeds by using a refined version of this idea and by giving an upper bound on the number these products.
\end{proof}

Consider $T > 0$.  Set $\ttv = 0$ in $\sum_{n = 0}^\infty g_n(2 \pi \text{\tt u})\, (4 \pi^2 \text{\tt v})^n$ and use condition $(3)$ of Definition~\ref{HSM} to see that $g_0$ is of moderate growth. The solution $\phi(\ttu)=0$ to the differential equation in $(c)$ of Theorem~\ref{thm:siegelmodularforms:harmonicfourierexpansions} gives $g_0(\ttu): = c\, \ttu^{1 - k}$ for some $c\in\C$.  Let $l_0(\ttu): = \ttu^{1 - k}$ and $l_n:=g_n$ as in $(c)$ of Theorem~\ref{thm:siegelmodularforms:harmonicfourierexpansions}.  We find that the hypotheses of Lemma~\ref{la:siegelmodularforms:coefficients_of_power_series_with_recursion} are satisfied (for details see~\cite{Raum-thesis}).  Choose $\kappa$ according to the second part of Lemma~\ref{la:siegelmodularforms:coefficients_of_power_series_with_recursion}, such that $\sum_{n = 0}^\infty l_n (\ttu) \left(\tfrac{\ttu}{\kappa}\right)^n$ is well-defined.  In particular, we can choose $\kappa$ such that $\sum_{n = 0}^\infty l_n (\ttu) \, \ttv^n$ grows rapidly as $\ttu \rightarrow \infty$, where $\ttv = \frac{\ttu}{\kappa}$.  The $M$-Whittaker function is another solution to the differential equation in $(c)$ of Theorem \ref{thm:siegelmodularforms:harmonicfourierexpansions}, but $M_{1 - k, (\sgn k)\, (k - \frac{3}{2})} (2 \tilde{\text{\tt u}})$ grows rapidly as $\tilde{\text{\tt u}} \rightarrow \infty$.  We conclude that the solutions to $g_0$ in $(c)$ of the remarks to Theorem~\ref{thm:laplacesiegelfourierexpansion} lead to rapidly growing Fourier coefficients $a(Y,T)$, which proves the case $T > 0$.

\vspace{1ex}

Consider the case $T$ indefinite.  We will need the following lemma.



\begin{lemma}
\label{la:siegelmodularforms:hypergeometriccoefficientgrowth}
Suppose that $k < 0$.  The ratio of the coefficient of $\ttv^n$ of the power series expansion
\begin{gather*}
  (\tfrac{\ttv}{4})^{\frac{3}{2} - k}
  \hypergeometric{2}{3} \bigl( 1, 2 - k; \tfrac{5}{2} - k, 2 - \tfrac{k}{2}, \tfrac{5 - k}{2};
                               \tfrac{\ttv}{4} \bigr)
\end{gather*}
and the coefficient of $\ttv^n$ of the power series expansion
\begin{gather*}
  (\tfrac{\ttv}{4})^{-\frac{k}{2}}
  \hypergeometric{1}{2} \bigl( \tfrac{1 - k}{2}; \tfrac{1}{2}, 1 - \tfrac{k}{2}; \tfrac{\ttv}{4} \bigr)
\end{gather*}
tends to zero as $n \rightarrow \infty$.

In particular, any linear combination of the first and the second generalized hypergeometric series in $(d)$ of the remarks to Theorem~\ref{thm:laplacesiegelfourierexpansion} grows rapidly as $\ttv \rightarrow \infty$.
\end{lemma}
\begin{proof}
The second generalized hypergeometric function of Lemma \ref{la:siegelmodularforms:hypergeometriccoefficientgrowth} equals, up to a polynomial,
\begin{gather*}
  (\tfrac{\ttv}{4})^{\frac{3}{2} - k}
  \hypergeometric{1}{2} \bigl( 2 - k; 2 - \tfrac{k}{2}, \tfrac{5 - k}{2};
                               \tfrac{\ttv}{4} \bigr)
\text{.}
\end{gather*}
This allows us to compute the ratio of the coefficients of $\ttv^{\td n}$, which tends to zero as ${\td n} \rightarrow \infty$.

The rapid growth of the linear combinations of the generalized hypergeometric series in Lemma \ref{la:siegelmodularforms:hypergeometriccoefficientgrowth} follows, since the coefficients of said linear combinations are almost all positive or almost all negative.
\end{proof}

We now argue that for every $k \in \ZZ$ the three fundamental solutions given in $(d)$ of the remarks to Theorem~\ref{thm:laplacesiegelfourierexpansion} lead to rapidly growing Fourier coefficients $a(Y,T)$.  If the solution is a Laurent polynomial, then this follows from Lemma \ref{la:siegelmodularforms:coefficients_of_power_series_with_recursion}.  If not, it follows by setting $\ttu = 0$ in $(d)$ of Theorem \ref{thm:laplacefourierexpansion} and the fact that nonpolynomial generalized hypergeometric series grow rapidly towards infinity.

If $k \ge 3$, then the space of solutions for $h_1$ in $(d)$ of the remarks to Theorem~\ref{thm:laplacesiegelfourierexpansion} is spanned by two polynomials and a generalized hypergeometric series.  In fact, the first and third solution given in $(d)$ of the remarks to Theorem~\ref{thm:laplacesiegelfourierexpansion} are, up to polynomials, multiples of each other.  This can be seen by analyzing the Laurent series expansion of both solutions with respect to $\ttv$ (see~\cite{Raum-thesis} for more details).  Any solution that occurs must be a linear combination of the two polynomials only, since otherwise, $a(Y, T)|_{\ttu = 0}$ grows rapidly.  On the other hand, nonvanishing polynomial solutions lead to rapidly growing $a(Y,T)$ by Lemma~\ref{la:siegelmodularforms:coefficients_of_power_series_with_recursion}, as in the case of $T > 0$.  Hence neither of the three solutions can occur.

We have to use a different argument if $k < 0$.  Lemma \ref{la:siegelmodularforms:hypergeometriccoefficientgrowth} shows that any nonzero linear combination of the second and third solution grows rapidly.  Lemma~\ref{la:siegelmodularforms:coefficients_of_power_series_with_recursion} allows us to exclude the first solution, which coincides, up to a polynomial, with a multiple of the third solution.  This yields the claim.

\vspace{1ex}

Finally, if $T>0$ and $k>3$, then we will employ the asymptotic behavior of $g_0$ to show that any possible nontrivial solution for $a(Y,T)$ is not in the kernel of $\xi_{\frac{1}{2}, k - \frac{1}{2}}^{(2)}$.  We need to consider $\phi$ to obtain the asymptotic behavior of $g_0$.   For generic $k$ the solutions for $\phi$ are the two generalized hypergeometric series
\begin{align*}
e^{-\tilde u} {\tilde u}^{k - 1} {}_1 {\rm F}_1 (1; 4 - 2 k; \tilde u)
\quad\text{and}\\
e^{-\tilde u} {\tilde u}^{2 - k} {}_1 {\rm F}_1 (2 k - 2; 2 k - 2; \tilde u)
\text{.}
\end{align*}

The following calculations can be performed with Sage \cite{sage}.  The Laurent series expansions of the solutions to $\phi$ around $\tilde u = 0$ yield Laurent series expansions of $\int {\td \ttu}^{-1} \phi(2 {\td \ttu}) \, d{\td \ttu}$, i.e., (up to additive constants) Laurent series expansions of $\psi$.  We can consider these Laurent series expansions as asymptotic expansions for $\psi$ as $u \rightarrow 0^+$.  Consequently, we may multiply the resulting expansions for $(\partial_Z \, \,a(Y, I_2) e^{2\pi i\tr X})\, e^{-2\pi i\tr X}$ by $\det Y^{k - \frac{1}{2}} \sim (\frac{u}{2})^{2k - 1}$.

Not all generalized hypergeometric series here are defined for integral $k$, but linear combinations admit analytic continuations (for details see~\cite{Raum-thesis}).  One finds that a linear combination of the generalized hypergeometric series above can only be in the kernel of $\xi_{\frac{1}{2}, k - \frac{1}{2}}^{(2)}$ if the limit of the asymptotic expansion of the linear combination of the corresponding $g_0$ tends to zero.  An inspection of the initial exponent of this expansion shows that this is not the case.

\end{proof}

\section{Harmonic skew-Maass-Jacobi forms}
\label{sec:skew-Jacobi}
The classical Jacobi forms in Eichler and Zagier \cite{EZ} are holomorphic functions.  
More generally, the Maass-Jacobi forms in Berndt and Schmidt \cite{BS}, Pitale \cite{Ameya}, and in \cite{BRR-H-Harmonic,B-R-Maass-Jacobi}
are real-analytic functions that are eigenfunctions of differential operators invariant under the action of the extended real Jacobi group.
Another important class of Jacobi forms are Skoruppa's \cite{Sko-mpim, Sko-Invent90} skew-holomorphic Jacobi forms, which are 
real-analytic in $\tau\in\HH$, holomorphic in $z\in\C$, and annihilated by the heat operator 
$$
L_m:=8 \pi im \partial_{\tau}-\partial_{zz}.
$$  
We now introduce necessary notation to define harmonic skew-Maass-Jacobi forms, which are real-analytic extensions of skew-holomorphic Jacobi forms.

Let $\Gamma^\rmJ:=\SL{2}(\Z)\ltimes \Z^2$ be the Jacobi group.  For fixed integers $k$ and $m$, define the following slash operator on functions $\phi:\HH \times\C\rightarrow\C$\,:
\begin{equation}
\label{skew-slash}
\begin{split}
& 
\Bigl(\phi\,\big|_{k,m}^{\sk} A \Bigr)(\tau,z)
\\ & 
:=\phi\left(\frac{a\tau+b}{c\tau+d},\frac{z+\lambda\tau+\mu}{c\tau+d}\right)(c\overline{\tau}+d)^{1-k}\,|c\tau+d|^{-1}\, e^{2\pi im\left(-\frac{c(z+\lambda\tau+\mu)^2}{c\tau+d}+\lambda^2\tau+2\lambda z\right)}
\end{split}
\end{equation}
for all $A=\left[\left(\begin{smallmatrix}a & b\\c & d\end{smallmatrix}\right), (\lambda, \mu)\right]\in\Gamma^\rmJ$.  Note that (\ref{skew-slash}) 
can be extended to an action $|^{\sk, \R}_{k,m}$ of the extended real Jacobi group on $\C^{\infty}\left(\HH\times\C\right)$.  The center of the universal enveloping algebra of the extended real Jacobi group is generated by a linear element and a cubic element, the {\it Casimir element}.  The linear element acts by scalars under $|^{\sk, \R}_{k,m}$ and the action of the Casimir element under $|^{\sk, \R}_{k,m}$ is given (up to the constant $8\pi i m\bigl(\frac{5}{8}+\frac{3(1-k)-(1-k)^2}{2}+1-2k\bigr)=8\pi i m\bigl(\frac{21}{8}-\frac{5k+k^2}{2}\bigr)$) by the following differential operator:
\begin{equation*}
\begin{split}
\mathcal{C}_{k,m}^{\sk} := & -2(\tau-\overline{\tau})^2\partial_{\overline{\tau}}L_m+(2k-1)(\tau-\overline{\tau})L_m \\ 
& +2(1-k)(\tau-\overline{\tau})\partial_{z\overline{z}}+ 2(\tau-\overline{\tau})(z-\overline{z})\partial_{zz\overline{z}}\\ & 
-16\pi i m(\tau-\overline{\tau})(z-\overline{z})\partial_{\tau\overline{z}}+8\pi i m(1-k)(z-\overline{z})\partial_{\overline{z}}\\ &+2(\tau-\overline{\tau})^2\partial_{\tau\overline{z}\overline{z}}+\left(4\pi im(z-\overline{z})^2+(\tau-\overline{\tau})\right)\partial_{\overline{z}\overline{z}}+ 2(\tau-\overline{\tau})(z-\overline{z})\partial_{z\overline{z}\overline{z}}\,.\\
\end{split}
\end{equation*}
In particular, $\mathcal{C}_{k,m}^{\sk}$ commutes with the action in (\ref{skew-slash}), i.e.,  if $A \in \Gamma^\rmJ$, then 
\begin{equation*}
 \left(\mathcal{C}_{k,m}^{\sk} \phi \right)\big|_{k,m}^{\sk} A=\mathcal{C}_{k,m}^{\sk} \left(\phi \big|_{k,m}^{\sk} A\right).
\end{equation*}

\vspace{1ex}

\begin{defn}
\label{sMJ}
A real-analytic function $\phi:\HH \times\C\rightarrow\C$ is a harmonic skew-Maass-Jacobi form of weight $k$ and index $m>0$ if the following conditions hold:

\vspace{0.5ex}

\begin{enumerate}
\item
For all $A\in\Gamma^\rmJ$, $\phi\,\big|_{k,m}^{\sk} A=\phi$.

\vspace{0.5ex}

\item
We have that $\mathcal{C}_{k,m}^{\sk}(\phi)=0$.

\vspace{1ex}

\item
We have that $\phi(\tau,z)=O\hspace{-0.5ex}\left(e^{ay} e^{2 \pi mv^2/y} \right)$ as $y\ra\infty$ for some $a>0$, and where $y=\mathrm{Im}(\tau)$ and $v=\mathrm{Im}(z)$.
\end{enumerate}

\vspace{1ex}

\noindent
We are especially interested in harmonic skew-Maass-Jacobi forms, which are holomorphic in $z$; we denote the space of such forms by $\widehat{\mathbb{J}}_{k,m}^{\sk}$.
\end{defn}

\vspace{1ex}

\begin{remarks}
\ 
\begin{enumerate}[(a)]
\item
One finds that every $\phi\in\widehat{\mathbb{J}}_{k,m}^{\sk}$ has a Fourier expansion of the form 
\begin{equation} 
\label{Fourier-skew}
\begin{split}
\phi(\tau,z)= & y^{\frac32-k}  \sum_{\substack{n, r \in\Z \\D=0  } } c^0(n,r)q^n\zeta^r\\ & 
+ \sum_{\substack{n, r \in\Z \\D \gg -\infty} } c^+(n,r)e^{-\frac{\pi Dy}{m}} q^n\zeta^r
+ \sum_{\substack{n, r \in\Z \\D \ll \infty   } } c^-(n,r)
H\hspace{-0.2ex}\left(\frac{\pi Dy}{2m}  \right) e^{-\frac{\pi Dy}{2m}} q^n\zeta^r.\\
\end{split}
\end{equation}
Here $D:=r^2-4mn$ and $H(w):=e^{-w}\int_{-2w}^{\infty}e^{-t}t^{\frac{1}{2}-k}dt$ converges 
for $k<\frac{3}{2}$ and has a holomorphic continuation in $k$ if $w\not=0$ and if $w<0$, 
then $H(w)=e^{-w}\, \Gamma(\frac{3}{2}-k,-2w)$ (see also page 55 of \cite{B-F-Duke04}).

\vspace{1ex}

\item
If $c^{0}(n,r)=0$ and $c^{-}(n,r)=0$ in (\ref{Fourier-skew}), then $\phi$ is a weak skew-holomorphic Jacobi form as in \cite{B-R-Maass-Jacobi}. 
If, in addition, $c^{+}(n,r)=0$ for all $D<0$ (resp. $D\leq 0$), then $\phi$ is a skew-holomorphic Jacobi form (resp. skew-holomorphic Jacobi cusp form) 
of weight $k$ and index $m$ as in \cite{Sko-mpim, Sko-Invent90}.  We denote the spaces of weak skew-holomorphic Jacobi forms and skew-holomorphic Jacobi forms, each of weight $k$ and index $m$, by $\rmJ_{k,m}^{\sk!}$ and $\rmJ_{k,m}^{\sk}$, respectively. 

\vspace{1ex}

\item The harmonic Maass-Jacobi forms in \cite{B-R-Maass-Jacobi} are real-analytic functions $\phi:\HH \times\C\rightarrow\C$ which are in the kernel of 
$C^{k,m}:=\frac{1}{8\pi i m}\left(y^{\frac{1}{2}-k}\mathcal{C}_{1-k,m}^{\sk}y^{k-\frac{1}{2}}+2k-1\right)$ and invariant under the usual Jacobi slash-operator $\big|_{k,m}$.  Recall that $\big|_{k,m}$ is as in (\ref{skew-slash}), except that $(c\overline{\tau}+d)^{1-k}\,|c\tau+d|^{-1}$ in (\ref{skew-slash}) is replaced by $\left(c\tau+d\right)^{-k}$.  Note also that $C^{\frac{1}{2},m}=\frac{1}{8\pi i m}\mathcal{C}_{\frac{1}{2},m}^{\sk}$.

\vspace{1ex}

\item 
Bruinier and Funke's differential operator $\xi_k$ plays an important role in the theory of harmonic weak Maass forms.  
The differential operator $\xi_{k,m}$ in \cite{B-R-Maass-Jacobi} is the corresponding operator for harmonic Maass-Jacobi forms, 
and there is also an analogous operator $\xi_{k,m}^{\sk}$ for harmonic skew-Maass-Jacobi forms.  Specifically, note that
$$
D_-^{\sk}:=\frac{y^2}{4\pi m}\,L_m
$$ 
is a ``lowering'' operator, i.e., if $\phi$ is a smooth function on $\HH\times\C$ and if $A\in\Gamma^\rmJ$, then
\begin{equation*}
\label{lowering}
\left(D_-^{\sk}\phi\right)\,\big|_{k-2,m}^{\sk} A =D_-^{\sk}\left(\phi\,\big|_{k,m}^{\sk} A\right) \,.
\end{equation*}
Set
\begin{equation}
\label{xi_J^{sk}}
\xi_{k,m}^{\sk}:=y^{k-\frac{5}{2}}D_-^{\sk}=\frac{y^{k-\frac{1}{2}}}{4\pi m}\,L_m.
\end{equation}
Then a direct computation shows that
$$
\xi_{k,m}^{\sk}:\widehat{\mathbb{J}}_{k,m}^{\sk}\rightarrow \rmJ_{3-k,m}^!,
$$
where $\rmJ_{k,m}^!$ denotes the space of weak Jacobi forms of weight $k$ and index $m$ (see also \cite{B-R-Maass-Jacobi}).  
The main results of \cite{B-R-Maass-Jacobi} can be extended to harmonic skew-Maass-Jacobi forms.  Specifically, 
one can define skew-Maass-Jacobi-Poincar\'{e} series, which are mapped under $\xi_{k,m}^{\sk}$ to holomorphic Jacobi-Poincar\'{e} series 
and which satisfy Zagier-type dualities when $k$ is replaced by $3-k$.
\end{enumerate}
\end{remarks}

\section{Kohnen's limit process}
\label{sec:limitprocess}

In this section, we will first employ Kohnen's work \cite{Ko-poincare} to find the limit (\ref{Kohnen's limit}) in case of the Poincar\'{e}-Eisenstein series $P_{k,s}$ in (\ref{Poincare-Eisenstein series}) for $s=0$ and $s=\frac{3}{2}-k$.  This will then allow us to perform the limit process for arbitrary $F \in \widehat{\mathbb{M}}_k$ and to prove Theorem~\ref{main}.



\begin{thm}
\label{thm:eisensteinlimit}
Let $\phi_m(\tau,z,y')$ be the $m$-th Fourier-Jacobi coefficient of $P_{k,s}(Z)$ as in (\ref{eq:fourierjacobiexpansion}).  
If $m>0$, then the limit in (\ref{Kohnen's limit}) exists for $s=0, k>3$ and for $s = \frac{3}{2} - k, k < 0$, and we have:
\begin{enumerate}[(a)]
\item If $s = 0$ and $k > 3$, then $y^{\frac{1}{2} - k\,} \mathcal{L}\bigl(\det Y^{k - \frac{1}{2}\,} \phi_m\bigr)\in \rmJ_{k,m}^{\sk}$.

\vspace{1ex}

\item If $s = \frac{3}{2} - k$ and $k < 0$, then $y^{\frac{1}{2} - k\,}\mathcal{L}\bigl(\det Y^{k-\frac{1}{2}\,} \phi_m\bigr)\in \widehat{\mathbb{J}}_{k,m}^{\sk}$.
\end{enumerate}
Moreover, the limits in $(a)$ and $(b)$ are not identically zero ($k$ is odd by assumption). 
\end{thm}
\begin{proof}
Let $k'\in\Z$ and let $s'\in\C$ such that $\mathrm{Re}(s')>\frac{3-k'}{2}$. Kohnen~\cite{Ko-poincare} considers
\begin{equation}
\label{Kohnen-Eisenstein series}
\mathcal{E}_{k',s'}(Z):=\sum_{M \in \Gamma_{\infty} \backslash \Gamma} \left((\det Y)^{s'} |_{(k',\,0)}\, M\right)(Z)
\text{,}
\end{equation}
and he points out (see p. 85 of \cite{Ko-poincare}) that applying the limit process (\ref{Kohnen's limit}) to $\mathcal{E}_{k',s'}$ yields a finite linear combination of Jacobi-Poincar\'{e} series of the form
\begin{equation}
\label{Poincare-skew}
P_{k',m,s'}(\tau,z) := \sum_{A\in\Gamma_{\infty}^\rmJ\backslash\Gamma^\rmJ}\, 
\bigl( y^{s'}\big|_{k',m}\, A\bigr)(\tau,z).
\end{equation}
Here $\Gamma_{\infty}^\rmJ:=\left\{\left[\left(\begin{smallmatrix}1 & \eta\\ 
0 &1 \end{smallmatrix}\right),(0,n)\right]\,\,|\,\,\eta,n\in\Z \right\}$ and $\big|_{k',m}$ is again the usual Jacobi slash-operator.

$(a)$ If $s'=k-\frac{1}{2}$ and $k'=1-k$ ($k>3$), then $\mathcal{E}_{1-k,k-\frac{1}{2}}(Z)=(\det Y)^{k-\frac{1}{2}}P_{k,0}(Z)$ and applying the limit process (\ref{Kohnen's limit}) gives a finite linear combination of the form
$$
P_{1-k,m,k-\frac{1}{2}}(\tau,z)=y^{k-\frac{1}{2}}\underbrace{\sum_{A\in\Gamma_{\infty}^\rmJ\backslash\Gamma^\rmJ}\, 
\bigl( 1\big|_{k,m}^{\sk}\, A\bigr)(\tau,z)}_{:=\phi(\tau,z)},
$$
where $\phi\in \rmJ_{k,m}^{\sk}$ is the usual skew-holomorphic Jacobi-Eisenstein series, which does not vanish ($k$ is odd).

$(b)$ If $s'=1$ and $k'=1-k$ ($k<0$), then $\mathcal{E}_{1-k,1}(Z)=(\det Y)^{k-\frac{1}{2}}P_{k,\frac{3}{2}-k}(Z)$ and applying the limit process (\ref{Kohnen's limit}) gives a finite linear combination of the form
$$
P_{1-k,m,1}(\tau,z)=y^{k-\frac{1}{2}}\underbrace{\sum_{A\in\Gamma_{\infty}^\rmJ\backslash\Gamma^\rmJ}\, 
\bigl( y^{\frac{3}{2}-k}\big|_{k,m}^{\sk}\, A\bigr)(\tau,z)}_{:=\psi(\tau,z)}.
$$
It is easy to check that $\psi\in \widehat{\mathbb{J}}_{k,m}^{\sk}$.  Finally, $\psi$ is not identically zero, since
$$
\xi_{k,m}^{\sk}(\psi)=(\tfrac{3}{2}-k)\sum_{A\in\Gamma_{\infty}^\rmJ\backslash\Gamma^\rmJ}\, 
\bigl( 1\big|_{3-k,m}\, A\bigr)(\tau,z)
$$
is a nonvanishing holomorphic Jacobi form of weight $3-k$ and index $m$.
\end{proof}

\vspace{1ex}

Now we give the proof of our main result.
\begin{proof}[Proof of Theorem~\ref{main}]
Let $F(Z)=\sum_{T} a(Y,T)e^{2 \pi i\,\tr(TX)}\in\widehat{\mathbb{M}}_k$ with Fourier-Jacobi expansion as in (\ref{eq:fourierjacobiexpansion}), and suppose that $\xi_{\frac{1}{2}, k-\frac{1}{2}}^{(2)}(F)=0$ if $k > 3$. Write $T= \left(\begin{smallmatrix}n & r \\ r & m\end{smallmatrix}\right)$ and assume that $m>0$.  Note that if the limit $\phi:=y^{\frac{1}{2} - k\,} \mathcal{L}\bigl(\det Y^{k - \frac{1}{2}\,} \phi_m\bigr)$ exists, then it follows easily that $\phi$ satisfies conditions $(1)$ and $(3)$ of Definition \ref{sMJ}.  Moreover, if $\rank (T) = 2$, then $a(Y,T)$ has only one fundamental solution by Theorem~\ref{thm:laplacesiegelfourierexpansion}.  

Consider the case $T>0$.  If $k<0$, then the Fourier coefficients $a(T)$ (for $T>0$) of the usual holomorphic Siegel-Eisenstein series of weight $3-k$ are nonzero and Proposition \ref{xi of Poincare series} implies that the Fourier coefficients $b(Y,T)$ ($T>0$) of the Poincaré-Eisenstein series $P_{k,\tfrac{3}{2} - k}$ are nonzero.  
Hence $a(Y,T)=\lambda\cdot b(Y,T)$ for some $\lambda\in\C$ and Theorem~\ref{thm:eisensteinlimit} yields the desired result.  If $k>3$, then $a(Y,T)=0$ due to the assumption $\xi_{\frac{1}{2}, k-\frac{1}{2}}^{(2)}(F)=0$.  Note that this assumption is necessary to our argument, since the Fourier coefficients $b(Y,T)$ of $P_{k,0}$ (for $k>3$) vanish for $T>0$ as can be seen from their integral representations in $\S 18$ of \cite{Maass} or from the fact that $\xi^{(2)}_{\frac{1}{2}, k - \frac{1}{2}} \big( P_{k, 0} \big) = 0$ (Proposition \ref{xi of Poincare series}).

Consider the case $T$ indefinite.  Theorem~\ref{thm:eisensteinlimit} asserts that Kohnen's limit process applied to $P_{k,s}$ yields a nonvanishing skew-holomorphic Jacobi form if $s=0, k>3$ and a nonvanishing harmonic skew-Maass Jacobi if $s=\tfrac{3}{2} - k, k<0$.  Thus, there exists an indefinite $T'=\left(\begin{smallmatrix}* & * \\ * & m'\end{smallmatrix}\right)$ with $m'>0$ such that the coefficient $a(Y,T')$ is a scalar multiple of a nonzero Fourier coefficient of the Poincaré-Eisenstein series $P_{k,0}$ if $k > 3$ or $P_{k,\tfrac{3}{2} - k}$ if $k < 0$, and Theorem~\ref{thm:eisensteinlimit} yields the desired result for this particular $T'$.  We have to show that the limit $y' \rightarrow \infty$ of $a(Y,T) \exp(2 \pi i m \tau')$ exists for all indefinite $T$.  Observe that every indefinite index $T$ with $m>0$ can be written as $T=H T' \trans H$ for some real, invertible, upper triangular matrix $H$.  One finds that the traces and determinants of $(\trans H Y H)T'$ and $Y(H T' \trans H)=YT$ are equal, and Theorem~\ref{thm:laplacesiegelfourierexpansion} implies that $a(\trans H Y H, T')=a(Y, H T' \trans H)=a(Y,T)$.   Furthermore, if $\left(\begin{smallmatrix} * & * \\ * & d\end{smallmatrix}\right)_{22}:=d$, then $T'_{22} (\trans H Z H)_{22} \sim (H T' \trans H)_{22} Z_{22}=m\tau'$ as $y'\rightarrow\infty$.  Hence,
\begin{equation*}
\begin{split}
\lim_{y' \rightarrow \infty} a(Y, T) \exp(2 \pi i m\tau') &
=
  \lim_{y' \rightarrow \infty} a(Y, H T' \trans H) \exp(2 \pi i (H T' \trans H)_{22} Z_{22})\\
& =
    \lim_{y' \rightarrow \infty} a(\trans H Y H, T') \exp(2 \pi i T'_{22} (\trans H Z H)_{22})\\
\end{split}
\end{equation*}
exist.

It remains to consider the case with $\rank(T)=1$, $T\geq 0$.  The explicit formula for $a(Y,T)$ in $(b)$ of Theorem~\ref{thm:laplacesiegelfourierexpansion} and the asymptotic behavior of the incomplete Gamma function (\ref{incomplete Gamma growth}) and the $W$-Whittaker function (\ref{Whittaker growth}) imply that
\begin{equation}
\label{rank 1 limit}
\begin{split}
& \hspace{1.5em} y^{\frac{1}{2}-k}\mathcal{L}\left(\det Y^{k - \frac{1}{2}\,} a(Y,T)e^{2\pi i(nx+2ru)}\right)\\
& =c_1\tfrac{(4\pi)^{1-k}}{m}y^{\frac{3}{2}-k}e^{2\pi i(n\tau+2rz)}+c_2\tfrac{(4\pi)^{\frac{1-k}{2}}}{m^{k-\frac{1}{2}}}e^{2\pi i(n\tau+2rz)},\\
\end{split}
\end{equation}
where $c_1,c_2\in\C$ are the constants in $(b)$ of Theorem~\ref{thm:laplacesiegelfourierexpansion}.  Observe that the right hand side of (\ref{rank 1 limit}) is in the kernel of $\mathcal{C}_{k,m}^{\sk}$ and we conclude that $y^{\frac{1}{2} - k\,} \mathcal{L}\bigl(\det Y^{k - \frac{1}{2}\,} \phi_m\bigr)\in\widehat{\mathbb{J}}_{k,m}^{\sk}$.  If, in addition, $k>3$, then $\xi_{\frac{1}{2}, k-\frac{1}{2}}^{(2)}(F)=0$ and hence $c_1=0$ by $(b)$ of Theorem~\ref{thm:laplacesiegelfourierexpansion}.  Finally, the second term on the right hand side of (\ref{rank 1 limit}) is in the kernel of the heat operator, i.e., if $k>3$, then the right hand side of (\ref{rank 1 limit}) is a Fourier coefficient of a skew-holomorphic Jacobi form.  We conclude that if $k>3$, then $y^{\frac{1}{2} - k\,} \mathcal{L}\bigl(\det Y^{k - \frac{1}{2}\,} \phi_m\bigr)\in \rmJ_{k,m}^{\sk}$.

\end{proof}



\vspace{2ex}

\noindent
{\it Acknowledgments}:
We are indebted to \"Ozlem Imamo\=glu for introducing the problem to us, and we thank her and Charles Conley for many helpful and inspiring conversations.

\bibliographystyle{acm}
\bibliography{Lit}

\begin{thebibliography}{10}

\bibitem{A-Saito-Kurokawa}
{\sc Andrianov, A.}
\newblock Modular descent and the {S}aito-{K}urokawa conjecture.
\newblock {\em Invent.\ {M}ath. {\bf 53}}, no.~3 (1979), 267--280.

\bibitem{BS}
{\sc Berndt, R., and Schmidt, R.}
\newblock {\em Elements of the representation theory of the {J}acobi group}.
\newblock \rm {P}rogr.\ {M}ath.\ {\bf 163}. Birkh\"auser, Basel, 1998.

\bibitem{Borel}
{\sc Borel, A.}
\newblock {\em {\it Introduction to automorphic forms}, {\rm in: {A}lgebraic
  groups and discontinuous subgroups ({P}roc. {S}ympos. {P}ure {M}ath.,
  {B}oulder, {C}olo., 1965)}}.
\newblock Amer.\ {M}ath.\ {S}oc., 1966, pp.~199--210.

\bibitem{BRR-H-Harmonic}
{\sc Bringmann, K., Raum, M., and Richter, O.}
\newblock Harmonic {M}aass-{J}acobi forms with singularities and a theta-like
  decomposition.
\newblock Preprint (2012).

\bibitem{B-R-Maass-Jacobi}
{\sc Bringmann, K., and Richter, O.}
\newblock {Z}agier-type dualites and lifting maps for harmonic {M}aass-{J}acobi
  forms.
\newblock {\em Adv.\ {M}ath. {\bf 225}}, no.~4 (2010), 2298--2315.

\bibitem{B-F-Duke04}
{\sc Bruinier, J., and Funke, J.}
\newblock On two geometric theta lifts.
\newblock {\em Duke {M}ath.\ {J}. {\bf 125}}, no.~1 (2004), 45--90.

\bibitem{DMZ}
{\sc Dabholkar, A., Murthy, S., and Zagier, D.}
\newblock Quantum black holes and mock modular forms.
\newblock Preprint (2011).

\bibitem{singular}
{\sc Decker, W., Greuel, G.-M., Pfister, G., and Sch{\"o}nemann, H.}
\newblock {\sc Singular} {3-1-1} --- {A} computer algebra system for polynomial
  computations.
\newblock http://www.singular.uni-kl.de.

\bibitem{Egu-Oog-Tac}
{\sc Eguchi, T., Ooguri, H., and Tachikawa, Y.}
\newblock Notes on the ${K}3$ surface and the {M}athieu group ${M}24$.
\newblock {\em Exp.\ {M}ath. {\bf 20}}, no.~1 (2011), 91--96.

\bibitem{EZ}
{\sc Eichler, M., and Zagier, D.}
\newblock {\em The theory of {J}acobi forms}.
\newblock Birkh\"auser, Boston, 1985.

\bibitem{Frei_1}
{\sc Freitag, E.}
\newblock {\em Siegelsche {M}odulfunktionen}.
\newblock Springer, {B}erlin, {H}eidelberg, {N}ew {Y}ork, 1983.

\bibitem{GNY-DiffGeom08}
{\sc G{\"o}ttsche, L., Nakajima, H., and Yoshioka, K.}
\newblock Instanton counting and {D}onaldson invariants.
\newblock {\em J.\ {D}ifferential {G}eom. {\bf 80}}, 3 (2008), 343--390.

\bibitem{GZ-Selecta98}
{\sc G{\"o}ttsche, L., and Zagier, D.}
\newblock Jacobi forms and the structure of {D}onaldson invariants for
  {$4$}-manifolds with {$b_+=1$}.
\newblock {\em Selecta {M}ath.\ ({N}.{S}.) {\bf 4}}, 1 (1998), 69--115.

\bibitem{I-R-RIMS10}
{\sc Imamo\=glu, {\"O}., and Richter, O.}
\newblock Differential operators and {S}iegel-{M}aass forms.
\newblock In {\em Automorphic forms, automorphic representations and related
  topics, 109-115}, {RIMS} {K}\^oky\^uroku {\bf{1715}}, Kyoto (2010).

\bibitem{Ko-problems}
{\sc Kohnen, W.}
\newblock Jacobi forms and {S}iegel modular forms: Recent results and problems.
\newblock {\em Enseign.\ {M}ath. (2) {\bf 39}\/} (1993), 121--136.

\bibitem{Ko-poincare}
{\sc Kohnen, W.}
\newblock Non-holomorphic {P}oincar\'e-type series on {J}acobi groups.
\newblock {\em J.\ {N}umber {T}heory {\bf 46}\/} (1994), 70--99.

\bibitem{plural}
{\sc Levandovskyy, V.}
\newblock {\sc Plural}, a non-commutative extension of {\sc singular}: past,
  present and future.
\newblock In {\em Mathematical software---{ICMS} 2006}, vol.~4151 of {\em
  Lecture Notes in Comput. Sci.} Springer, Berlin, 2006, pp.~144--157.

\bibitem{Luke-I}
{\sc Luke, Y.}
\newblock {\em The special functions and their approximations, {V}ol. {I}}.
\newblock Mathematics in {S}cience and {E}ngineering, {V}ol. 53. Academic
  {P}ress, New {Y}ork, 1969.

\bibitem{Maass-diff}
{\sc Maass, H.}
\newblock Die {D}ifferentialgleichungen in der {T}heorie der {S}iegelschen
  {M}odulfunktionen.
\newblock {\em Math.\ {A}nn. {\bf 126}\/} (1953), 44--68.

\bibitem{Maass}
{\sc Maass, H.}
\newblock {\em Siegel's {M}odular forms and {D}irichlet series}, vol.~{\bf 216}
  of {\em \rm Lecture {N}otes in {M}ath.}
\newblock Springer, 1971.

\bibitem{Maass-Spezial-I}
{\sc Maass, H.}
\newblock \"{U}ber eine {S}pezialschar von {M}odulformen zweiten {G}rades.
\newblock {\em Invent.\ {M}ath. {\bf 52}}, no.~1 (1979), 95--104.

\bibitem{Maass-Spezial-II}
{\sc Maass, H.}
\newblock \"{U}ber eine {S}pezialschar von {M}odulformen zweiten {G}rades.
  {II}.
\newblock {\em Invent.\ {M}ath. {\bf 53}}, no.~3 (1979), 249--253.

\bibitem{Maass-Spezial-III}
{\sc Maass, H.}
\newblock \"{U}ber eine {S}pezialschar von {M}odulformen zweiten {G}rades.
  {III}.
\newblock {\em Invent.\ {M}ath. {\bf 53}}, no.~3 (1979), 255--265.

\bibitem{Ono-Mal}
{\sc Malmendier, A., and Ono, K.}
\newblock ${SO}(3)$-{D}onaldson invariants of $\mathbb{CP}^2$ and mock theta
  functions.
\newblock Preprint (2009).

\bibitem{Naka-Japan82}
{\sc Nakajima, S.}
\newblock On invariant differential operators on bounded symmetric domains of
  type {${\rm IV}$}.
\newblock {\em Proc.\ {J}apan {A}cad.\ {S}er.\ {A} {M}ath.\ {S}ci. {\bf 58}},
  no.~6 (1982), 235--238.

\bibitem{Niwa-Nagoya}
{\sc Niwa, S.}
\newblock On generalized {W}hittaker functions on {S}iegel's upper half space
  of degree {$2$}.
\newblock {\em Nagoya {M}ath.\ {J}. {\bf 121}\/} (1991), 171--184.

\bibitem{Ameya}
{\sc Pitale, A.}
\newblock Jacobi {M}aa\ss \, forms.
\newblock {\em Abh.\ {M}ath.\ {S}em.\ {U}niv.\ {H}amburg {\bf 79}\/} (2009),
  87--111.

\bibitem{Raum-thesis}
{\sc Raum, M.}
\newblock {\em Dual weights in the theory for harmonic {S}iegel modular forms}.
\newblock PhD thesis, University of {B}onn, {G}ermany, 2012.

\bibitem{Sko-mpim}
{\sc Skoruppa, N.-P.}
\newblock Developments in the theory of {J}acobi forms.
\newblock {\em Acad.\ {S}ci.\ {USSR}, {I}nst.\ {A}ppl.\ {M}ath.,
  {K}habarovsk\/} (1990), 167--185.

\bibitem{Sko-Invent90}
{\sc Skoruppa, N.-P.}
\newblock Explicit formulas for the {F}ourier coeffcients of {J}acobi and
  elliptic modular forms.
\newblock {\em Invent.\ {M}ath. {\bf 102}}, no.~3 (1990), 501--520.

\bibitem{sage}
{\sc Stein, W., et~al.}
\newblock {\em {S}age {M}athematics {S}oftware ({V}ersion 4.6.2)}.
\newblock The Sage Development Team, 2010.
\newblock {\tt http://www.sagemath.org}.

\bibitem{Z-SK}
{\sc Zagier, D.}
\newblock {\em {\it {S}ur la conjecture de {S}aito-{K}urokawa}, {\rm in:
  {S}eminar on {N}umber {T}heory, {P}aris 1979-80}}.
\newblock \rm {P}rogr. {M}ath. {\bf 12}. Birkh\"auser, 1981, pp.~371--394.

\bibitem{Zwe-thesis}
{\sc Zwegers, S.}
\newblock Mock theta functions, Ph.D. thesis, Universiteit Utrecht, The
  Netherlands, 2002.

\end{thebibliography}

\end{document}